\documentclass[12pt]{amsart}
\usepackage{amsmath,amssymb,amscd,graphicx}
\usepackage[colorlinks=true, pdfstartview=FitV, linkcolor=blue, citecolor=blue, urlcolor=blue]{hyperref}
\usepackage{mathpazo}

\theoremstyle{plain}
\newtheorem{theorem}{Theorem}
\newtheorem*{conjecture}{Conjecture}
\newtheorem*{lemma}{Lemma}

\newtheorem{proposition}{Proposition}

\theoremstyle{remark}

\DeclareMathOperator{\csch}{csch}

\newcommand{\tre}{\text{Re}}
\newcommand{\tim}{\text{Im}}
\newcommand{\pp}{{\prime\prime}}

\theoremstyle{definition}
\newtheorem*{hypothesis}{Hypothesis D}

\begin{document}
\title{X-Ray of Zhang's eta function}
\author{Jeffrey Stopple}
\begin{abstract}
Study of the level curves $\tre(\eta(s))=0$ and $\tim(\eta(s))=0$, for $\eta(s)=\pi^{-s/2}\Gamma(s/2)\zeta^\prime(s)$ gives a new classification of the zeros of $\zeta(s)$ and of $\zeta^\prime(s)$.  Numerical evidence indicates that the statistics of the gaps (between zeros of $\zeta$), or distance from the critical line (for zeros of $\zeta^\prime$) is related to the classification.  Theorem 5 gives the full conjecture of Soundararajan for the zeros we classify as type 2.  We assume the Riemann Hypothesis throughout.
\end{abstract}
\email{stopple@math.ucsb.edu}

\maketitle

 \baselineskip=16pt

\subsection*{Introduction}
In \cite{Zhang}, Zhang named the function
\[
\eta(s)=\pi^{-s/2}\Gamma(s/2)\zeta^\prime(s).
\]
This function has an interesting property with respect to the zeros of $\zeta(s)$ on the critical line:
\begin{lemma} (\cite[(1.6)]{Levinson} or  \cite[Lemma 1]{Zhang}) Suppose $t>7$.  Then  we have $\zeta(1/2+it)=0$  if and only if $\tre(\eta(1/2+i t))=0$.
\end{lemma}
The lemma makes the level curves for $\eta(s)$ of interest.  In \cite{Arias}, Arias-de-Reyna used the terminology \lq X-ray\rq\  for the level curves $\tre(\zeta(s))=0$ and $\tim(\zeta(s))=0$.  
Arias-de-Reyna used thick and thin lines for the level curves  $\tre(\zeta(s))=0$ and $\tim(\zeta(s))=0$ respectively.  We use color for the level curves of $\eta(s)$, and in addition we color separately based on the sign of the component which is not $0$.  Thus the colors in Figure \ref{F:1} can be interpreted as follows:

\begin{itemize}
\item[Red:] $\tim(\eta(s))=0$ and $\tre(\eta(s))>0$
\item[Green:] $\tre(\eta(s))=0$ and $\tim(\eta(s))>0$
\item[Cyan:] $\tim(\eta(s))=0$ and $\tre(\eta(s))<0$
\item[Purple:] $\tre(\eta(s))=0$ and $\tim(\eta(s))<0$
\end{itemize}

\begin{figure}
\begin{center}
\includegraphics[scale=1, viewport=0 0 350 300,clip]{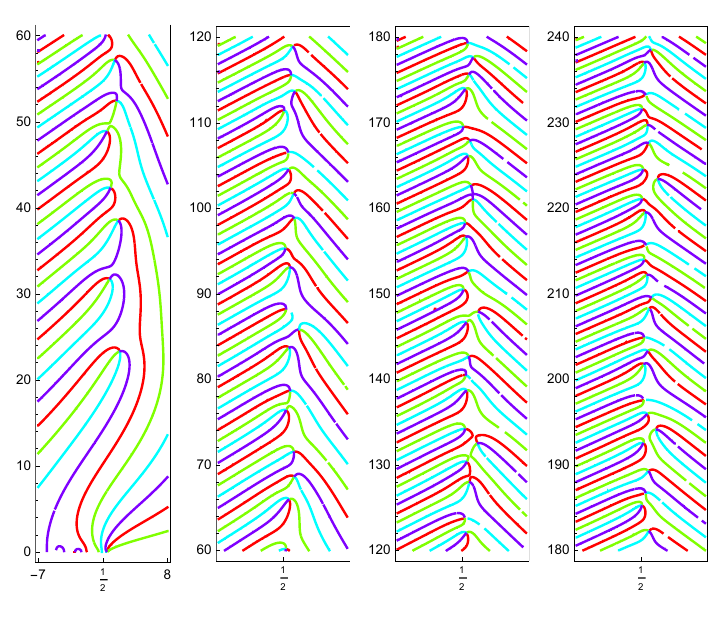}
\caption{Level curves for $\eta(s)$, $-7\le\sigma\le 8$, $0\le t\le 240$}\label{F:1}
\end{center}
\end{figure}

\noindent (In \emph{Mathematica} these colors are \textsc{Hue[0]}, \textsc{Hue[1/4]}, \textsc{Hue[1/2]}, and \textsc{Hue[3/4]} respectively.)

Throughout we assume the Riemann Hypothesis.  We also need to assume the following
\begin{hypothesis}   The level curves $\tre(\eta(s))=0$ are differentiable.  This is automatic except at isolated points where $\eta^\prime(s)=0$, so we are really assuming that when $\eta^\prime(s)=0$, $\arg(\eta(s))\ne\pm\pi/2$.  This prevents the level curves from branching.   Hypothesis D is  plausible because the two arguments $\pm\pi/2$ are measure zero on the unit circle, while the zeros of $\eta^\prime(s)$ form a countable set.
\end{hypothesis}

For shorthand when referring to \lq the zeros\rq\  of $\zeta(s)$ we mean the nontrivial zeros only.  The Riemann zeros $\rho=1/2+i\gamma$ of $\zeta(s)$ occur where the green and purple contours cross the critical line.  The zeros $\rho^\prime$ of $\zeta^\prime$ are visible everywhere four colors come together (exclusive of the double pole at $s=1$.)

\begin{figure}
\begin{center}
\includegraphics[scale=1, viewport=0 0 350 300,clip]{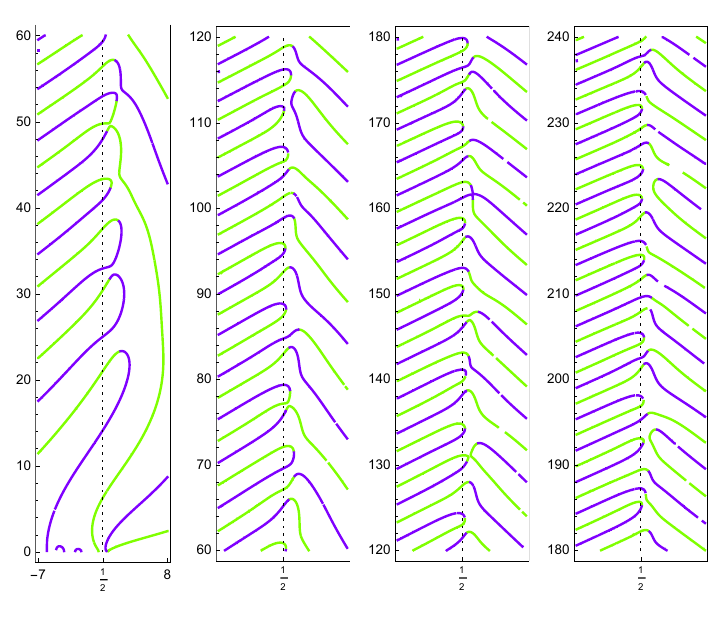}
\caption{Level curves for $\eta(s)$, $-7\le\sigma\le 8$, $0\le t\le 240$}\label{F:2}
\end{center}
\end{figure}

\noindent
Here's a summary of the sections of the paper:
\begin{enumerate}
\item[\S\ref{S:classify}] Classification the zeros of $\zeta(s)$ and $\zeta^\prime(s)$ into different types by means of the level curves.
\item[\S\ref{S:data}]   Investigation the asymptotics of the types, and examination of the data for approximately $10^5$ zeros of $\zeta^\prime$ near $T=10^6$.
\item[\S\ref{S:Spira}] A canonical bijection between the zeros $\rho^\prime$ of $\zeta^\prime(s)$ with $\tre(s)>1/2$, and the zeros $\rho^\pp$ of $\zeta^{\prime\prime}(s)$ with $\tre(s)>1/2$.
\item[\S\ref{S:Marden}]  Adaptation of a theorem of Marden.  The location of $\rho^\pp$ relative to $\rho^\prime$.
\item[\S\ref{S:application}] Application of the classification: the conjecture of Soundararajan for type 2 zeros.
\item[\S\ref{S:IZL}] Appendix I: The \lq Improved Zhang-Ge Lemma\rq, a study of $h^\prime/h+\zeta^{\prime\prime}/\zeta^\prime$ explicit enough to remove the constraint \lq for all large  $\gamma$\rq.
\item[\S \ref{S:numerics}] Appendix II: Description of the algorithm used to obtain the data.
\end{enumerate}

\section{Classification of zeros}\label{S:classify}

\begin{theorem} With the usual indexing $\gamma_1<\gamma_2<\ldots$ of the imaginary parts of the zeros of $\zeta(s)$, every odd indexed zero lies on a contour $\tim(\eta(s))<0$.  Every even indexed zero lies on a contour $\tim(\eta(s))>0$. 
\end{theorem}
\begin{proof}
This follows from the Improved Zhang-Ge Lemma below, which says that as $t$ increases, the argument of $\eta(1/2+it)$ decreases by exactly $\pi$ between consecutive zeros.  A \emph{Mathematica} calculation of $\eta(1/2+i\gamma_1)$ determines the parity of all the zeros.
\end{proof}

\subsubsection*{Zeros of $\zeta^\prime(s)$}
\begin{itemize}
\item[Type 0:] We will say a zero $\rho^\prime$ of $\zeta^\prime(s)$ is of \textsc{type 0} if \emph{neither} of the level curves $\tre(\eta(s))=0, \tim(\eta(s))>0$ and $\tre(\eta(s))=0, \tim(\eta(s))<0$ exiting $\rho^\prime$ cross the critical line $\sigma=1/2$.
\item[Type 1:] We will say a zero $\rho^\prime$ of $\zeta^\prime(s)$ is of \textsc{type 1} if \emph{exactly one} of the level curves $\tre(\eta(s))=0, \tim(\eta(s))>0$ and $\tre(\eta(s))=0, \tim(\eta(s))<0$ exiting $\rho^\prime$ crosses the critical line $\sigma=1/2$.
\item[Type 2:] We will say a zero $\rho^\prime$ of $\zeta^\prime(s)$ is of \textsc{type 2} if the level curves $\tre(\eta(s))=0, \tim(\eta(s))>0$ and $\tre(\eta(s))=0, \tim(\eta(s))<0$ exiting $\rho^\prime$ \emph{both} cross the critical line $\sigma=1/2$.
\end{itemize}
(To be completely precise, \lq crosses the critical line\rq\ above should really be replaced with \lq crosses the critical line above $t=7$\rq, since there is a curve originating in the double pole at $s=1$, which crosses the critical line below $t=7$ but does not correspond to a zero of $\zeta(s)$.  The Lemmas do not apply in this region.)\ \   These zeros could be further classified according to what the other two contours are doing, but I don't (yet) see the utility.

\subsubsection*{Zeros of $\zeta(s)$}
\begin{itemize}
\item[Type 1:] We will say a zero $\rho=1/2+i\gamma$ of $\zeta(s)$ is of \textsc{type 1} if the level curve $\tre(\eta(s))=0$ on which it lies, terminates in a zero $\rho^\prime$ which is of type 1.
\item[Type 2:] We will say a zero $\rho=1/2+i\gamma$ of $\zeta(s)$ is of \textsc{type 2} if the level curve $\tre(\eta(s))=0$ on which it lies, terminates in a zero $\rho^\prime$ which is of type 2.
\end{itemize}

Figure \ref{F:2} is Figure \ref{F:1} with the $\tim(\eta(s))=0$ curves removed, to see more easily the zeros of $\zeta(s)$ (curve crosses the critical line) and $\zeta^\prime(s)$ (curves of different colors meet) and their types.  When both branches form a loop to the left, it is type 2.  When they loop to the right, it is type 0.  If the two colors extend in opposite directions without looping, it is type 1.  In Figure \ref{F:2}, the first four zeros of $\zeta^\prime(s)$ have type 2; the next four alternate between types 1 and 2.  The first zero of type 0 occurs at height about 113, with another at height about 132.  At height about 161 we have two consecutive zeros of type 1, but from the way the graphics are imported into Latex one can not tell, looks like it might be a type 2 and type 0.  It seems a zero of $\eta^\prime(s)$ is nearby.  The breaks in the curves are an artifact of the \emph{Mathematica} \textsc{ContourPlot} command; they could be eliminated by setting the parameters to sample more points.

\begin{theorem}\label{Thm:classification}
Every Riemann zero is of either type 1 or type 2.  Thus we have a canonical mapping from the zeros of $\zeta(s)$ to those of $\zeta^\prime(s)$, which is two to one on the type 2 zeros, and one to one on the type 1 zeros.  Zeros of $\zeta^\prime(s)$ of type 0 are precisely those not in the image of this mapping.  The Riemann zeros of type 2 are canonically grouped in pairs.
\end{theorem}
\begin{proof}
All this is clear except the first statement, which says the contours which cross the critical line from the left must terminate in exactly one zero of $\zeta^\prime(s)$.  Since we are assuming Hypothesis D, the alternatives we must rule out is continuation of the contour on to the right, or looping back to the left.  

For the first possibility, note that the contour $\arg(\eta(s))=\pi/2$ (resp. $\arg(\eta(s))=-\pi/2$) does not exist in isolation; it is part of a continuum which deform smoothly as the argument is varied.  But the argument of $\eta(s)$ is increasing (as one moves up vertically in the plane) for $\tre(s)>3$, but decreasing for $\tre(s)<0$.  They can only cross over each other where the argument of $\eta(s)$ is undefined, at a zero $\rho^\prime$.

The second possibility is ruled out by the Improved Zhang Lemma, which says that the argument of $\eta(s)$ decreases monotonically as one moves up the critical line.
\end{proof}

\section{Asymptotics and data}\label{S:data}
Let
\[
N_1(T)=\sharp\left\{\text{type 1 zeros }1/2+i\gamma\,|\,0<\gamma<T\right\}.
\]
NB: This is a nontraditional notation for the meaning of $N_1(T)$.
Let
\[
N_2(T)=
\sharp\left\{\text{pairs of type 2 zeros }1/2+i\gamma^-,1/2+i\gamma^+\,|\,0<\gamma^+<T\right\}.
\]
We have classically 
\begin{equation}\label{Eq:1}
N_1(T)+2N_2(T)=\frac{T}{2\pi}\log\left(\frac{T}{2\pi}\right)-\frac{T}{2\pi}+O\left(\log T\right).
\end{equation}
For $j=0,1,2$, let
\[
N_j^\prime(T)=\sharp\left\{\text{zeros }\rho^\prime=\beta^\prime+i\gamma^\prime\text{ of type }j\,|\,0<\gamma^\prime<T\right\}.
\]
NB: The ${}^\prime$ here does not indicate a derivative with respect to $T$.
The Theorem above implies $N_1(T)=N_1^\prime(T)$ and $N_2(T)=N_2^\prime(T)$.
Thus we have from \cite{Ber}:
\begin{equation}\label{Eq:2}
N_0^\prime(T)+N_1(T)+N_2(T)=\frac{T}{2\pi}\log\left(\frac{T}{4\pi}\right)-\frac{T}{2\pi}+O\left(\log T\right).
\end{equation}
Subtracting (\ref{Eq:2}) from  (\ref{Eq:1}) gives
\begin{equation}\label{Eq:3}
N_2(T)-N_0^\prime(T)=\frac{T}{2\pi}\log\left(2\right)+O\left(\log T\right).
\end{equation}
Subtracting (\ref{Eq:1}) from twice (\ref{Eq:2}) gives
\begin{equation}\label{Eq:4}
N_1(T)+2N_0^\prime(T)=\frac{T}{2\pi}\log\left(\frac{T}{8\pi}\right)-\frac{T}{2\pi}+O\left(\log T\right).
\end{equation}
These two estimates immediately give the following
\begin{theorem}   There are infinitely many type 2 zeros of $\zeta(s)$, and thus also of $\zeta^\prime(s)$.  At least one of the other types of zeros of $\zeta^\prime(s)$ is infinite in number.
\end{theorem}

Up to $T=1000$ it is possible to plot the level curves and classify the 537 zeros of $\zeta^\prime$ by hand: there are 75 of type 0 (14\%), 280 of type 1 (52\%), and 182 of type 2 (34\%).  
An algorithm was developed (see \S \ref{S:numerics}) to classify the $108043$ zeros of $\zeta^\prime$ near $T=10^6$, previously computed by Farmer for use in \cite{Duenez}.  One finds 23902 zeros of type 0 (22.1\%), 53621 of type 1 (49.6\%) and 30520 of type 2 (28.2\%).

We make the following conjecture
\begin{conjecture}  For some constant $C$, possibly equal $0$
\begin{align*}
N_0^\prime(T)=&\frac{1}{8\pi} T\log\left(\frac{T}{4\pi}\right)+\left(C-\frac{\log 2}{4\pi} \right)\cdot T+O(\log T)\\
N_1(T)=&\frac{1}{4\pi} T\log\left(\frac{T}{4\pi}\right)-\left(2C+\frac{1}{2\pi}\right)\cdot T+O(\log T)\\
N_2(T)=&\frac{1}{8\pi}T\log\left(\frac{T}{4\pi}\right)+\left(C+\frac{\log 2}{4\pi} \right)\cdot T+O(\log T).
\end{align*}
\end{conjecture}
The coefficients of the $T\log(T)$ terms in the conjecture are based on the heuristic that each of the two contours $\tre(\eta(s))=0,$ $\tim(\eta(s))>0$ and $\tre(\eta(s))=0,$  $\tim(\eta(s))<0$ emanating from a zero $\rho^\prime$ of $\zeta^\prime$ has equal chance of exiting the critical strip to the left or to the right.  And this is in rough agreement with the numerical evidence.  The lower order terms are the simplest expression in agreement with (\ref{Eq:2}), (\ref{Eq:3}) and (\ref{Eq:4}). 

\section{Zeros of $\zeta^{\prime\prime}(s)$}\label{S:Spira}

Spira, in \cite{Spira} was the first to observe that zeros of successive higher order derivatives of the Riemann zeta function seem to cluster along roughly horizontal lines.  He wrote \lq\lq The zeros of $\zeta^{\prime\prime}$ have imaginary part almost exactly equal to those of $\zeta^\prime$, and lie to the right of them.\rq\rq\ \   (See Figure 1 from his paper, or Figure \ref{F:moanswer} below.)\ \   The following explains Spira's observation.

\begin{figure}
\begin{center}
\includegraphics[scale=1.7, viewport=0 0 150 290,clip]{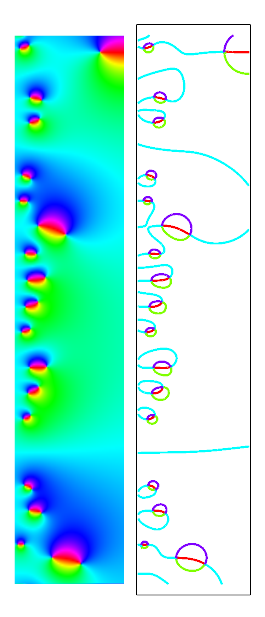}
\caption{$\arg(\zeta^{\prime\prime}(s)/\zeta^\prime(s))$, for $1/2\le\sigma\le 5/2$ and $10^6\le t\le 10^6+10$}\label{F:moanswer}
\end{center}
\end{figure}

\begin{theorem}\label{Thm4}
The level curves $\arg(\zeta^{\prime\prime}/\zeta^\prime(s))=0$ connect each zero  of $\zeta^\prime(s)$ with $\tre(s)>1/2$ to a zero of $\zeta^{\prime\prime}(s)$ with $\tre(s)>1/2$ (typically to the right), giving a canonical bijection between these two sets.  The same holds for zeros higher derivatives $\zeta^{(k)}(s)$ and $\zeta^{(k+1)}(s)$.
\end{theorem}
\begin{proof}
As $\sigma\to +\infty$, $\zeta^{\prime\prime}/\zeta^\prime(s)\to -\log(2)$.  Meanwhile, on the critical line, $\zeta^{\prime\prime}/\zeta^\prime(s)$ will be dominated by terms in the sum $\sum_{\rho^\prime} 1/(s-\rho^\prime)$ with $\rho^\prime$ near $s$, and this implies the real part of $\zeta^{\prime\prime}/\zeta^\prime(s)$ will again be negative.

Now fix a zero $\rho^\prime$ of $\zeta^\prime(s)$, and consider the level curves with 
$\arg=0$ exiting the pole at $\rho^\prime$.   By the above observations, this contour can't cross the critical line, nor extend too far into the right half plane.  The only possibility is that it terminates.  To finish the argument we must be be sure that a contour 
$\arg=0$ can not connect a zero of $\zeta^{\prime\prime}(s)$ to another zero nor a pole (i.e. zero of $\zeta^\prime(s)$) to another pole.  But this contour is the inverse image under $\zeta^\pp/\zeta^\prime$ of the positive real axis, connecting $0$ to $\infty$ on the Riemann sphere.  It can only connect zeros to poles.

Since
$$
\frac{\zeta^{\prime\prime}}{\zeta^\prime}(s)=\frac{1}{s-\rho^\prime}+O(1),
$$
the contour $\arg=0$ has to exit the pole to the right, and so the zeros of $\zeta^{\prime\prime}$ will be to the right of the zeros of $\zeta^\prime$.  Similarly, a contour with $\arg(\zeta^{\prime\prime}/\zeta^\prime(s))=0$ terminating in a zero of $\zeta^{\prime\prime}(s)$, when followed backwards, must originate in a pole, i.e., a zero of $\zeta^\prime(s)$.

This same argument works for higher derivatives as well.
\end{proof}

Figure \ref{F:moanswer} shows, for $1/2\le\sigma\le 5/2$ and $10^6\le t\le 10^6+10$, the argument of $\zeta^{\prime\prime}/\zeta^\prime(s)$ on the left, and only those curves corresponding to the four coordinate axes on the right.  The proof references the red curves.

\section{Adaptation of a Theorem of Marden}\label{S:Marden}

This section is inspired by the results in \cite{Marden1, Marden2}, which express the logarithmic derivative of an entire function $f$  as a sum over poles (zeros of $f$) weighted by rational expressions in a fixed set of zeros of $f$ and $f^\prime$.   Marden's proof of \cite[Theorem 2.1]{Marden1} via the Cauchy Integral Formula can be generalized to $f(s)=\zeta^\prime(s)$, but in fact  \cite[Theorem 2.1]{Marden1}  and \cite[Theorem 2.1]{Marden2} can be proved simply by a partial fraction decomposition and taking linear combinations of the  Hadamard logarithmic derivative, as shall see.

Following the notation of  \cite{Yildirim1}, we denote the real zeros of $\zeta^\prime(s)$ as $-a_n$, with $a_n\in (2n,2n+2)$.  In fact, \cite[Lemma 1]{Yildirim2},
\[
-a_n=-2n-2+\frac{1}{\log(n)}+O\left(\frac{1}{\log^2(n)}\right).
\]
A generic non-real zero of $\zeta^\prime(s)$ will be denoted as $\lambda^\prime$ in this subsection.  The following proposition will be useful in studying applications of type 2 zeros in the next section.

\begin{proposition}\label{P:Marden}
Fix a complex zero $\rho^\pp=\beta^\pp+i\gamma^\pp$ of $\zeta^\pp(s)$ with $\beta^\pp>1/2$.  With $s=\sigma+it$ in a vertical strip $a\le\sigma\le b$ we have
\begin{equation}\label{Eq:Marden}
\frac{\zeta^\pp}{\zeta^\prime}(s)=\frac{\log(\gamma^\pp/t)}{2}
+
\sum_{\lambda^\prime}\frac{\rho^\pp-s}{(s-\lambda^\prime)(\rho^\pp-\lambda^\prime)}
+O\left(\frac{1}{t}\right)+O\left(\frac{1}{\gamma^\pp}\right).
\end{equation}
The sum is uniformly convergent on compact sets.
\end{proposition}
\begin{proof}
The starting point is the partial fraction representation
\begin{multline}\label{Eq:new2}
\frac{\zeta^\pp}{\zeta^\prime}(s)=\frac{\zeta^\pp}{\zeta^\prime}(0)-2-\frac{2}{s-1}+\\
\sum_n\left(\frac{1}{s+a_n}-\frac{1}{a_n}\right)+\sum_{\lambda^\prime}\left(\frac{1}{s-\lambda^\prime}+\frac{1}{\lambda^\prime}\right)
\end{multline}
which follows from the Hadamard theory.
From \eqref{Eq:new2} subtract $0=\zeta^\pp/\zeta^\prime(\rho^\pp)$ to obtain
\begin{multline*}
\frac{\zeta^\pp}{\zeta^\prime}(s)=-\frac{2}{s-1}+\frac{2}{\rho^\pp-1}\\
+\sum_n\left(\frac{1}{s+a_n}-\frac{1}{\rho^\pp+a_n}\right)+\sum_{\lambda^\prime}\frac{\rho^\pp-s}{(s-\lambda^\prime)(\rho^\pp-\lambda^\prime)}.
\end{multline*}
Add and subtract $\psi(s/2)/2$ where the digamma function $\psi(s)=\Gamma^\prime/\Gamma(s)$.  From the series representation for the digamma function we see that
\begin{multline*}
\frac{\zeta^\pp}{\zeta^\prime}(s)=-\frac{1}{2}\psi(s/2)-\frac{C}{2}+\frac{1}{s}-\frac{2}{s-1}+\frac{2}{\rho^\pp-1}\\
+\sum_n\left(\frac{1}{2n}-\frac{1}{s+2n}+\frac{1}{s+a_n}-\frac{1}{\rho^\pp+a_n}\right)+\sum_{\lambda^\prime}\frac{\rho^\pp-s}{(s-\lambda^\prime)(\rho^\pp-\lambda^\prime)}.
\end{multline*}
We regroup the terms to obtain
\begin{multline*}
\frac{\zeta^\pp}{\zeta^\prime}(s)=-\frac{1}{2}\psi(s/2)-\frac{C}{2}+\frac{1}{s}-\frac{2}{s-1}+\frac{2}{\rho^\pp-1}+\\
\sum_n\frac{2n-a_n}{(s+2n)(s+a_n)}+
\sum_n\left(\frac{1}{2n}-\frac{1}{\rho^\pp+a_n}\right)+
\sum_{\lambda^\prime}\frac{\rho^\pp-s}{(s-\lambda^\prime)(\rho^\pp-\lambda^\prime)}.
\end{multline*}
\begin{lemma}
We have
\[
\frac{1}{s}-\frac{2}{s-1}+\sum_n\frac{2n-a_n}{(s+2n)(s+a_n)}=O\left(\frac{1}{t}\right)
\]
\end{lemma}
\begin{proof}
The numerator $2n-a_n$ of the summand is $O(1)$, while
\begin{multline*}
\sum_n\frac{1}{(s+2n)(s+a_n)}\ll |s|^2\sum_n\frac{1}{(4n^2+t^2)^2}+\\
|s|\sum_n\frac{4n}{(4n^2+t^2)^2}+\sum_n\frac{4n^2}{(4n^2+t^2)^2}.
\end{multline*}
The first and last sum on the right have complicated closed forms in terms of $\sinh$, $\cosh$, $\coth$, and $\csch$, while the middle sum is expressed in terms of $\psi^\prime$, the derivative of the digamma function.  Including the leading factors $|s|^2$, $|s|$ and $1$, each is $O(1/t)$.
\end{proof}
We add
\[
0=\frac{1}{2}\psi(\rho^\pp/2)+\frac{C}{2}-\frac{1}{\rho^\pp}+\sum_n\left(\frac{1}{\rho^\pp+2n}-\frac{1}{2n}\right)
\]
to see that
\begin{multline*}
\frac{\zeta^\pp}{\zeta^\prime}(s)=\frac{\psi(\rho^\pp/2)-\psi(s/2)}{2}
+\sum_{\lambda^\prime}\frac{\rho^\pp-s}{(s-\lambda^\prime)(\rho^\pp-\lambda^\prime)}
\\
-\frac{1}{\rho^\pp}+\frac{2}{\rho^\pp-1}+\sum_n\left(\frac{1}{\rho^\pp+2n}-\frac{1}{\rho^\pp+a_n}\right)+O\left(\frac{1}{t}\right).
\end{multline*}
Via the lemma,
\[
-\frac{1}{\rho^\pp}+\frac{2}{\rho^\pp-1}
+
\sum_n\left(\frac{1}{\rho^\pp+2n}-\frac{1}{\rho^\pp+a_n}\right)=O\left(\frac{1}{\gamma^\pp}\right).
\]
Stirling's formula gives
\[
\frac{\psi(\rho^\pp/2)-\psi(s/2)}{2}=\frac{1}{2}\log(\gamma^\pp/t)+O\left(\frac{1}{t}\right)+O\left(\frac{1}{\gamma^\pp}\right).
\]
\end{proof}

As a consequence of  Proposition \ref{P:Marden} and Stirling's formula, we note that for
$F(t)=-\tre\left(\eta^\prime/\eta\left(\frac12+i t\right)\right)$ as in Appendix I below, we have
\begin{multline}\label{Eq:newF(t)}
F(t)=-\frac{1}{2}\log(\gamma^\pp/2\pi)-
\sum_{\lambda^\prime}\tre\left(\frac{\rho^\pp-s}{(s-\lambda^\prime)(\rho^\pp-\lambda^\prime)}\right)\\
+O\left(\frac{1}{t}\right)+O\left(\frac{1}{\gamma^\pp}\right).
\end{multline}

Let $\rho^\prime$ denote the zero of $\zeta^\prime(s)$ canonically associated  via Theorem \ref{Thm4} to $\rho^\pp$.  Subtracting \eqref{Eq:FanGe} below from \eqref{Eq:newF(t)} gives
\begin{theorem}\label{Thm:fund}
\begin{equation}\label{Eq:fund}
\tre\left(\frac{1}{\rho^\pp-\rho^\prime}\right)+\sum_{\lambda^\prime\ne\rho^\prime}\tre\left(\frac{1}{\rho^\pp-\lambda^\prime}\right)=\frac{\log(\gamma^\pp/\pi)}{2}+O\left(\frac{1}{\gamma^\pp}\right).
\end{equation}
\end{theorem}
(The $O(1/t)$ term drops out as the rest of the expression is independent of $t$.)
\begin{figure}
\begin{center}
\includegraphics[scale=1.1, viewport=0 0 150 500,clip]{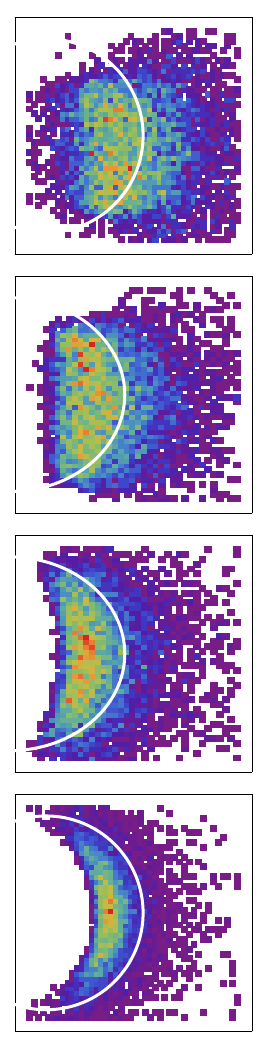}
\caption{}\label{F:10}
\end{center}
\end{figure}

Although we are interested below primarily in application of type 2 zeros, \eqref{Eq:fund} is valid for zeros $\rho^\prime$ of any type.  Theorem \ref{Thm:fund} is a fundamental identity that relates the location of the $\rho^\pp$ associated to $\rho^\prime$ to the location  of all the  $\lambda^\prime\ne\rho^\prime$.  If $\rho^\prime$ is  far from the  critical line, since $\rho^\pp$ still further to the right, most $\lambda^\prime$ contribute a positive term to the sum.  This means that
\[
\tre\left(\frac{1}{\rho^\pp-\rho^\prime}\right)<\frac{\log(\gamma^\pp/\pi)}{2}\approx\frac{\log(\gamma^\prime)}{2}.
\]
Since the level curves $\tre(1/(z-\rho^\prime))=c$ are circles with center $\rho^\prime+1/2c$ and  radius $1/2c$, we expect that $\rho^\prime$ and $\rho^\pp$ lie on opposite sides of a circle of radius larger than $1/\log(\gamma^\prime)$.  

On the other hand, if $\rho^\prime$ is close to the critical line, with  $\rho^\pp$ lying to the right of $\rho^\prime$ we expect there to be $\lambda^\prime$ with $\tre(1/(\rho^\pp-\lambda^\prime)$ positive as well as negative, so there is cancellation in the sum.  So we expect that when $\rho^\prime$ is close to the critical line,
\[
\tre\left(\frac{1}{\rho^\pp-\rho^\prime}\right)\approx\frac{\log(\gamma^\pp/\pi)}{2}\approx\frac{\log(\gamma^\prime)}{2},
\]
and so $\rho^\prime$ and $\rho^\pp$ lie on opposite sides of a circle of radius approximately $1/\log(\gamma^\prime)$.\label{page12}  (One can see both phenomena in Figure \ref{F:moanswer}.)

Figure \ref{F:10} shows data for the 30520 type 2 zeros $\rho^\prime$ near $T=10^6$: for each of the quartiles of $(\beta^\prime-1/2)\log(\gamma^\prime)$, a density histogram of the position of the canonically associated $\rho^{\pp}$ relative to $\rho^\prime+1/\log(\gamma^\prime)$, scaled by $\log(\gamma^\prime)$.  Red denotes the most points in a bin, purple the fewest.  With this normalization the circle is the unit circle, shown in white.  One sees more or less random behavior for the highest quartile (top).  As we go to lower quartiles  for $(\beta^\prime-1/2)\log(\gamma^\prime)$, the zeros $\rho^{\pp}$ are both less likely to be near $\rho^\prime$, and more likely to be inside the circle.

\section{Application of type 2 zeros}\label{S:application}
The horizontal distribution of the zeros of $\zeta^\prime$ has been studied by many authors since Levinson and Montgomery \cite{LandM}.
\begin{figure}
\begin{center}
\includegraphics[scale=1.3, viewport=0 0 310 180,clip]{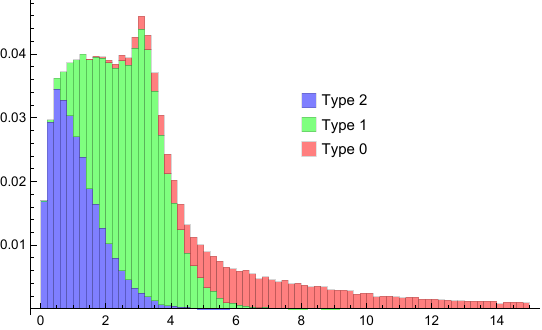}
\caption{$(\beta^\prime-1/2)\log(\gamma^\prime)$ for  $10^5$ zeros near $T=10^6$.}\label{F:3}
\end{center}
\end{figure}
Figure \ref{F:3} shows the histogram of 
$
(\beta^\prime-1/2)\log(\gamma^\prime)
$
for zeros of type 0, 1, and 2 separately.   This is the same data as Figure 5 in \cite{Duenez}, now separated by types.  In \cite{Duenez}, the authors write \emph{\lq\lq We would like to know the underlying cause of the curious \lq second bump\rq in the distribution of zeros of derivatives [of characteristic polynomials of matrices]... In figure 5 we find a similar shape for the distribution of zeros of $\zeta^\prime$.\rq\rq}\ \  Interestingly, the histograms analogous to Figure \ref{F:3} for the three types separately each show only a single peak; it is the interplay between them that causes the second bump.

In \cite{FarmerKi}, Farmer and Ki show that if $\zeta^\prime(s)$ has sufficiently many zeros close to the
critical line, then $\zeta(s)$ has many closely spaced zeros. This gives them a condition on the zeros of
$\zeta^\prime(s)$ which implies a lower bound of the class numbers of imaginary quadratic
fields.
One sees in Figure \ref{F:3} the type 2 zeros closest to the critical line, and the type 0 zeros the furthest.  In fact the median value of $(\beta^\prime-1/2)\log(\gamma^\prime)$ for the type two zeros is $0.98$; the other quartiles are $1.58$ and $0.54$.  Recalling the average value of $\beta^\prime$ is $1/2+\log\log(\gamma^\prime)/\log(\gamma^\prime)$, we can rescale by $\log\log(10^6)$ to see the median for type 2 zeros on this scale is $0.37$.  For comparison, the median for type 0 zeros on this scale is $2.97$.

This strongly motivates the further study of the types, and in particular, the type 2 zeros.  Corresponding to the type 2 zeros of $\zeta^\prime(s)$ in the numerical data, we have 30520 pairs of canonically associated type 2 zeros $1/2+i\gamma^-$, $1/2+i\gamma^+$ of $\zeta(s)$.
Figure \ref{F:4} shows the normalized gap 
\[
(\gamma^+-\gamma^-)\cdot \frac{\log(\gamma^\prime)}{2\pi};
\]
most are less than the average and 26\% are less than half the average.  The colors indicate for the  contributions from the different quartiles of $(\beta^\prime-1/2)\log(\gamma^\prime)$ (green the highest through red the lowest); one sees that when $\rho^\prime$ is closer to the critical line, the normalized gaps are smaller.

\begin{figure}
\begin{center}
\includegraphics[scale=1.3, viewport=0 0 310 160,clip]{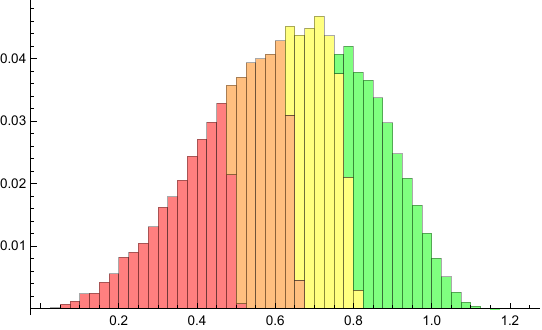}
\caption{Normalized gaps between type 2 zeros near $T=10^6$.}\label{F:4}
\end{center}
\end{figure}

Of interest is the conjecture of Soundararajan \cite{Sou}.  A pair of type 2 zeros of $\zeta(s)$ is canonically identified with a type 2 zero of $\zeta^\prime(s)$: they all lie on the same level curve $\tre(\eta(s))=0$. Since the data indicate that type 2 zeros of $\zeta(s)$ have smaller gaps, and type 2 zeros of $\zeta^\prime(s)$ are closer to the critical line, it is natural to ask if one can show
\begin{equation}\label{Eq:varSound}
\liminf (\beta^\prime -1/2)\log\gamma^\prime = 0 \Rightarrow \liminf (\gamma^+-\gamma^-)\log \gamma^\prime=0
\end{equation}
when the $\liminf$ are both restricted to the subsequence of type 2 zeros. 

We can investigate \eqref{Eq:varSound} via a study of the curvature of the level curve.  With $\eta(s)=u+iv$, the formula for the curvature $\kappa$ of the level curve $u(\sigma,t)=0$ may be found in \cite[\S 3]{Goldman}.  Via the Cauchy-Riemann equations, one sees \cite{Jerrard}
\begin{equation}\label{Eq:kappa1}
\kappa=|\eta^\prime|\cdot \tre(\eta^{\prime\prime}/\eta^{\prime\,2}).
\end{equation}

\begin{figure}
\begin{center}
\includegraphics[scale=1.3, viewport=0 0 310 160,clip]{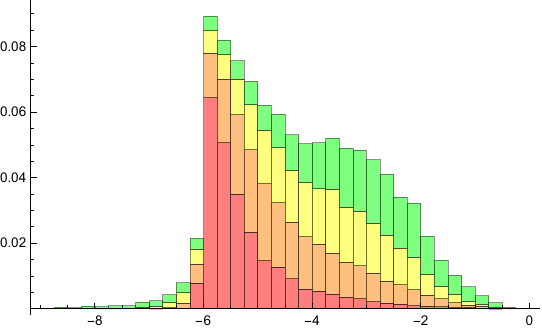}
\caption{Curvatures $\kappa$ for type 2 zeros near $T=10^6$.}\label{F:5}
\end{center}
\end{figure}

Because the level curve is not parametrized, one needs to be careful about orientation.  Equation (\ref{Eq:kappa1}) naturally orients the curve so the outward normal vector is in the direction of increasing real part $u$, which is in the direction of the red curves in Figure \ref{F:1}.  We will refer to this at the \textsc{canonical} orientation. In the numerical data, we have instead introduced a sign to orient the curves locally near a type 2 zero of $\zeta^\prime$ with increasing $t$.  This is possible precisely because the type 2 level curves cross the critical line twice.  We will refer to this as the \textsc{visual} orientation.  With this orientation, all but three of the 30520 data points have negative curvature, that is, curving to the left as $t$ increases.  Figure \ref{F:5} show the (visual) curvature of $u=0$ at $\rho^\prime$ for the type 2 zeros.  Again, the colors indicate the contributions from the different quartiles of $(\beta^\prime-1/2)\log(\gamma^\prime)$ (green the highest through red the lowest); one sees that those $\rho^\prime$ closer to the critical line tend to be  more curved.

(Type 1 curves only have the  canonical orientation, but since the type 0 curves cross the line $\tre(s)=3$ twice, they can also be given a visual orientation.)

To study the curvature, we first need two auxiliary results on the location of $\gamma^\prime$ relative to $\gamma^+$, $\gamma^-$, and on $\theta=\arg(\eta^\prime(\rho^\prime))$.

We claim that $\gamma^\prime$ is very near to $t_0=(\gamma^++\gamma^-)/2$ when either $\beta^\prime-1/2$ is small or $\gamma^+-\gamma^-$ is  small.  In fact, borrowing the notation of  \cite[p.13-14]{Stopple.Lehmer} we introduce $\Delta, t_0$, $Y$, and $\lambda$ defined by\label{page13}
\begin{gather*}
\Delta=\gamma^+-\gamma^-,\qquad \lambda=\log(t_0/2\pi),\\
\rho^{\pm}=1/2+i(t_0\pm \Delta/2),\qquad
\rho^\prime=\beta^\prime+i(t_0+Y)
\end{gather*}
so $\gamma^\prime=t_0+Y$.
We  rescale with 
\[
x=(\beta^\prime-1/2)\lambda,\qquad y=Y\lambda,\qquad\delta=\Delta\lambda/2\pi.
\]
In \cite[p.13-14]{Stopple.Lehmer} we developed series expansions
\begin{gather*}
x(\delta)=\frac{\pi^2}{4}\left(1-\frac{\log(\pi)}{\lambda}\right)\delta^2+O(\delta^4),\\
y(\delta)=\frac{\pi^2}{2\lambda}\left(\frac{\pi}{4}+\sum_{\rho\ne\rho^\pm}\frac{1}{t_0-\gamma}\right)\,\, \delta^2+O(\delta^4).
\end{gather*}
In the  first, we estimate $\delta^2$ in terms of $x$ and plug into the second.  We neglect the sum over $\rho\ne\rho^\pm$, which should show significant cancellation.  (For an individual example of a $\rho^\prime$, there may be an imbalance with more nearby $\rho$ above $\rho^+$ than below $\rho^-$ or vice versa.  But as we will be considering an infinite sequence of type 2 pairs, the only  way the result could fail is if all but finitely many of the pairs showed such an imbalance, which is not plausible.)\ \   Converting back to the original variables we get
\begin{equation}\label{Eq:last}
\gamma^\prime=t_0+O\left((\gamma^+-\gamma^-)^2\right)),\,\,\text{ and }\,\,   \gamma^\prime = t_0+O\left(\frac{\beta^\prime-1/2}{\log(\gamma^\prime)}\right).
\end{equation}

We will next investigate the angles $\theta$, by comparing to the argument of $\eta(1/2+i\gamma^\prime)$.  We observe that the argument of $\eta(1/2+it)$ changes from $-\pi/2$ to $\pi/2$ (or the reverse) as $t$ increases from $\gamma^-$ to $\gamma^+$, so the argument of $\eta(1/2+it_0)$ will be very near to either $0$ or $\pi$, and so will the argument of $\eta(1/2+i\gamma^\prime)$.  Since $\eta(\rho^\prime)=0$, the argument at $\rho^\prime$ is not defined, but there is a limiting value along the horizontal line $\sigma+i\gamma^\prime$ as $\sigma$ approaches $\beta^\prime$ from the left.  Consider the Taylor expansion of $\eta$ at $\rho^\prime$:
\[
\eta(\sigma+i\gamma^\prime)=h\zeta^{\prime\prime}(\rho^\prime)(\sigma-\beta^\prime)+O\left(\sigma-\beta^\prime\right)^2,
\]
so
\[
\frac{\eta(\sigma+i\gamma^\prime)}{\sigma-\beta^\prime}=h\zeta^{\prime\prime}(\rho^\prime)+O\left(\sigma-\beta^\prime\right).
\]
Because $\sigma-\beta^\prime$ is negative and real for $1/2<\sigma<\beta^\prime$, we see that
\begin{equation}\label{Eq:thetaest}
\lim_{\sigma\to\beta^\prime}\arg\left(\eta(\sigma+i\gamma^\prime)\right)=\arg\left(h\zeta^{\prime\prime}(\rho^\prime)\right)+\pi\bmod 2\pi.
\end{equation}
In the ($\Leftarrow$) half of the proof of Theorem \ref{Thm5} below, we show the change in the argument from the critical line to this limiting value is $o(1)$, by integrating $d \arg(\eta(\sigma+i\gamma^\prime))/d\sigma$ from $\sigma=1/2$ to $\sigma=\beta^\prime$.  Note the shift by $\pi$ modulo $2\pi$: when the argument of $\eta(1/2+i\gamma^\prime)$ is very near to $0$ (resp.\ $\pi$),  \eqref{Eq:thetaest} implies $\theta=\arg(h\zeta^{\prime\prime}(\rho^\prime))$ is near $\pi$ (resp.\ $0$) modulo $2\pi$.   

The data are shown in Figure \ref{F:6}.  As before, the colors indicate the contributions from the various quartiles of $(\beta^\prime-1/2)\log(\gamma^\prime)$; one sees that those $\rho^\prime$ closer to the critical line tend to have $\theta$ more strongly cluster around $0$ and $\pi$.  (When $\theta$ is near $0$, the canonical orientation is the same as the visual orientation; when $\theta$ is near $\pi$, the canonical orientation is the opposite of the visual orientation.)

\begin{figure}
\begin{center}
\includegraphics[scale=1.4, viewport=0 0 380 205,clip]{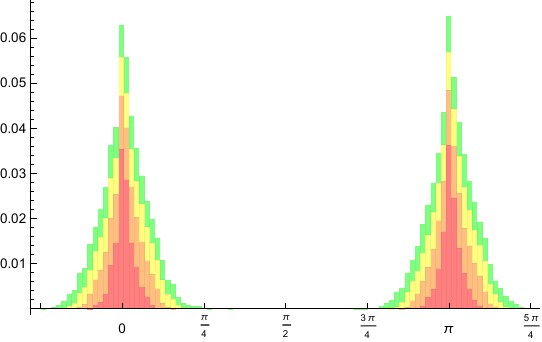}
\caption{$\arg(\eta^\prime(\rho^\prime))$ for type 2 zeros near $T=10.^6$.}\label{F:6}
\end{center}
\end{figure}

We can now investigate more fully the  significance of the curvature.
The osculating circle for the curve $\tre(\eta)=0$ at $\rho^\prime$ has radius $R=1/|\kappa|$.  Based on our study of the angles $\theta$, we will assume that $\rho^\prime$ is the leftmost point on the osculating circle.  Then elementary geometry shows the length of the chord this circle cuts from the critical line is
\[
2\left((\beta^\prime-1/2)(2/|\kappa|+1/2-\beta^\prime)\right)^{1/2}.
\]

The inequality 
\[
2\left((\beta^\prime-1/2)(2/|\kappa|+1/2-\beta^\prime)\right)^{1/2}<2^{3/2}\left(\frac{\beta^\prime-1/2}{|\kappa|}\right)^{1/2}.
\]
shows that, along a subsequence with 
\[
(\beta^\prime-1/2)\log(\gamma^\prime)\to 0,
\]
any lower bound of the form
\begin{equation}\label{Eq:needestimate}
\log(\gamma^\prime)\ll |\kappa|
\end{equation}
will force the length of the chord, multiplied by $\log(\gamma^\prime)$, to tend to $0$.  Since the error between the level curve and the osculating circle is of cubic order, we see that (\ref{Eq:needestimate}) will force
\[
(\gamma^+-\gamma^-)\log(\gamma^\prime)\to 0
\]
as well, which will prove \eqref{Eq:varSound}.  

The expression (\ref{Eq:kappa1}) for $\kappa$ simplifies when evaluated at a zero $\rho^\prime$ of $\zeta^\prime(s)$.  Write $h(s)=\pi^{-s/2}\Gamma(s/2)$, and suppressing the $\rho^\prime$, we have
\[
\eta^\prime=h\zeta^{\prime\prime},\qquad \eta^{\prime\prime}=2h^\prime\zeta^{\prime\prime}+h\zeta^{\prime\prime\prime}.
\]
Recall $\theta$ denote $\arg(\eta^\prime)$, so
\[
\frac{|\eta^\prime|}{\eta^\prime}=\exp(-i\theta),
\]
and the curvature at $\rho^\prime$ reduces to
\begin{equation}\label{Eq:kappa2}
\kappa=\tre\left(\exp(-i\theta)\left(\frac{2h^\prime}{h}+\frac{\zeta^{\prime\prime\prime}}{\zeta^{\prime\prime}}\right)\right).
\end{equation}

As we saw above, $\theta$ is very near $0$ or $\pi$, so $\exp(-i\theta)$ is very near $\pm1$, and does not contribute significantly to $|k|$. 

The next proposition relates the remaining parameters determining the curvature of the level curve at one zero $\rho^\prime$ of $\zeta^\prime(s)$ to the locations of all the other zeros $\lambda^\prime\ne\rho^\prime$ of $\zeta^\prime(s)$.

\begin{proposition} With $\rho^\pp$ a zero of $\zeta^\pp(s)$ as above, let $\rho^\prime=\beta^\prime+i\gamma^\prime$ be the canonically associated zero of $\zeta^\prime(s)$.  Then
\begin{multline}\label{Eq:Prop13}
\tre\left(2\frac{h^\prime}{h}(\rho^\prime)+\frac{\zeta^{\prime\prime\prime}}{\zeta^{\pp}}(\rho^\prime)\right)=\\
-\log(2)-2\sum_{\lambda^\prime\ne\rho^\prime}\tre\left(\frac{1}{\lambda^\prime-(1/2+i\gamma^\prime)}\right)\\
+O\left(\beta^\prime-1/2\right)+O\left(\frac{1}{\gamma^\prime}\right)+O\left(\frac{1}{\gamma^\pp}\right).
\end{multline}
\end{proposition}
\begin{proof}
We evaluate $\zeta^\pp/\zeta^\prime(s)$ at $s=1/2+i\gamma^\prime$ using Proposition \ref{P:Marden}, and also via a Laurent expansion at $\rho^\prime$:
\[
\frac{\zeta^\pp}{\zeta^\prime}(1/2+i\gamma^\prime)=\frac{1}{1/2-\beta^\prime}+\frac{\zeta^{\prime\prime\prime}}{2\zeta^{\pp}}(\rho^\prime)+O\left(\beta^\prime-1/2\right).
\]
Stirling's formula applied to $2h^\prime/h(\rho^\prime)$ and taking real parts shows that
\begin{multline}
\tre\left(2\frac{h^\prime}{h}(\rho^\prime)+\frac{\zeta^{\prime\prime\prime}}{\zeta^{\pp}}(\rho^\prime)\right)=
\log(\gamma^\pp/\pi)-\log(2)
+\\
2\sum_{\lambda^\prime\ne\rho^\prime}\tre\left(\frac{\rho^\pp-(1/2+i\gamma^\prime)}{((1/2+i\gamma^\prime)-\lambda^\prime)(\rho^\pp-\lambda^\prime)}\right)-\tre\left(\frac{2}{\rho^\pp-\rho^\prime}\right)\\
+O\left(\beta^\prime-1/2\right)+O\left(\frac{1}{\gamma^\prime}\right)+O\left(\frac{1}{\gamma^\pp}\right).
\end{multline}
We then use \eqref{Eq:fund} to replace
\[
\log(\gamma^\pp/\pi)-\tre\left(\frac{2}{\rho^\pp-\rho^\prime}\right)
\quad
\text { with  }
\quad
2\sum_{\lambda^\prime\ne\rho^\prime}\tre\left(\frac{1}{\rho^\pp-\lambda^\prime}\right).
\]
\end{proof}

We now come (at last) to the main result.
\begin{theorem}\label{Thm5}  For type 2 zeros,
\[
\liminf (\beta^\prime -1/2)\log\gamma^\prime = 0 \Leftrightarrow \liminf (\gamma^+-\gamma^-)\log \gamma^\prime=0
\]
\end{theorem}
\begin{proof} ($ \Rightarrow$)
Consider a sequence of type 2 triples $\rho^+$, $\rho^\prime$, $\rho^-$, with $(\beta^\prime-1/2)\log(\gamma^\prime)\to 0$.
By (\ref{Eq:needestimate}), (\ref{Eq:kappa2}), and (\ref{Eq:Prop13}), it suffices to show
\begin{equation}
\log(\gamma^\prime)\ll \sum_{\lambda^\prime\ne\rho^\prime}\tre\left(\frac{1}{\lambda^\prime-(1/2+i\gamma^\prime)}\right).
\end{equation}
Let $d^+=\gamma^+-\gamma^\prime$, $d^-=\gamma^\prime-\gamma^-$, $\Delta=d^++d^-=\gamma^+-\gamma^-$.  If $\Delta \log(\gamma^\prime)\to 0$ we are already done, so we may as well assume there is some $\epsilon>0$ so that
\[
\Delta \log(\gamma^\prime)>\epsilon.
\]
Based on \eqref{Eq:last}, we may as well assume $d^+, d^->\Delta/3$.
Let
\begin{align*}
G^+(t)=&\sum_{\substack{\lambda^\prime\\|\lambda^\prime-\rho^+|>d^+}}\tre\left(\frac{1}{\lambda^\prime-(1/2+it)}\right)\\
G^-(t)=&\sum_{\substack{\lambda^\prime\\|\lambda^\prime-\rho^-|>d^-}}\tre\left(\frac{1}{\lambda^\prime-(1/2+it)}\right)\\
G(t)=&\sum_{\lambda^\prime}\tre\left(\frac{1}{\lambda^\prime-(1/2+it)}\right)
\end{align*}
When  $|\lambda^\prime-\rho^+|>d^+$, we have
\[
|\lambda^\prime-(1/2+i\gamma^\prime)|\le|\lambda^\prime-\rho^+|+|\rho^+-(1/2+i\gamma^\prime)|<2|\lambda^\prime-\rho^+|,
\]
and so
\[
G^+(\gamma^+)/4 < G^+(\gamma^\prime).
\]
Similarly
\[
G^-(\gamma^-)/4< G^-(\gamma^\prime),
\]
so
\begin{multline*}
G^+(\gamma^+)/4+G^-(\gamma^-)/4< G^+(\gamma^\prime)  + G^-(\gamma^\prime)\\
\le 2\sum_{\lambda^\prime\ne\rho^\prime}\tre\left(\frac{1}{\lambda^\prime-(1/2+i\gamma^\prime)}\right).
\end{multline*}
Now if $|\lambda^\prime-\rho^+|\le d^+$ and $\lambda^\prime\ne\rho^\prime$, then $\tim(\lambda^\prime)>\tim(\rho^+)$, and so
\[
|\lambda^\prime-\rho^+|<|\lambda^\prime-\rho^-|,\quad\text{ and }\quad\tre\left(\frac{1}{\lambda^\prime-\rho^-}\right)<\tre\left(\frac{1}{\lambda^\prime-\rho^+}\right).
\]
Similarly, if $|\lambda^\prime-\rho^-|\le d^-$ and $\lambda^\prime\ne\rho^\prime$, then
\[
\tre\left(\frac{1}{\lambda^\prime-\rho^+}\right)<\tre\left(\frac{1}{\lambda^\prime-\rho^-}\right).
\]
Thus
\begin{multline*}
\left(G(\gamma^+)-\tre\left(\frac{1}{\rho^\prime-\rho^+}\right)\right)/8+\left(G(\gamma^-)-\tre\left(\frac{1}{\rho^\prime-\rho^-}\right)\right)/8
\\
<G^+(\gamma^+)/4+G^-(\gamma^-)/4.\
\end{multline*}
We have $(G(\gamma^+)+G(\gamma^-))/8\sim\log(\gamma^\prime)/8$ via \eqref{Eq:FanGe}, \eqref{Eq:ZLemma3}.  Next
\begin{multline*}
-\tre\left(\frac{1}{\rho^\prime-\rho^+}\right)>-\frac{\beta^\prime-1/2}{(\beta^\prime-1/2)^2+\Delta^2/9}=\\
-\frac{(\beta^\prime-1/2)\log(\gamma^\prime)}{((\beta^\prime-1/2)\log(\gamma^\prime))^2+(\Delta\log(\gamma^\prime))^2/9}\cdot\log(\gamma^\prime)\\
>-\frac{(\beta^\prime-1/2)\log(\gamma^\prime)}{((\beta^\prime-1/2)\log(\gamma^\prime))^2+\epsilon^2/9}\cdot\log(\gamma^\prime).\\
\end{multline*}
With numerator
\[
(\beta^\prime-1/2)\log(\gamma^\prime)\to 0
\]
and denominator
\[
((\beta^\prime-1/2)\log(\gamma^\prime))^2+\epsilon^2/9\to \epsilon^2/9, 
\]
we see that eventually
\[
-\tre\left(\frac{1}{\rho^\prime-\rho^+}\right)>-\delta\cdot\log(\gamma^\prime)
\]
for any  $\delta>0$, in particular for $\delta=1/4$ which suffices.  The term $\tre(1/(\rho^\prime-\rho^-))$  is handled similarly.

($ \Leftarrow$)
Consider a sequence of type 2 pairs $\rho^+$, $\rho^-$, with 
\[
(\gamma^+-\gamma^-)\log((\gamma^++\gamma^-)/2))\to 0.
\]
We know from \cite[Theorem  2]{FanGe2}, that 
for any
$v<0.4$ the following holds: For all sufficiently large $\gamma^+$, $\gamma^-$ with $\Delta< v/\log t_0$, (notation as in page \pageref{page13}) the box
$$
\{s=\sigma+it: \frac{1}{2}<\sigma<\frac{1}{2}+\frac{v^2}{4\log t_0},
\gamma^-\le t\le \gamma^+\}
$$
contains exactly one zero $\rho^\prime$  of $\zeta^\prime(s)$.  Because \cite{Zhang} proved this direction of Soundararajan's conjecture, the desired implication
is simpler, but not trivial.  We just need to confirm for this type 2 pair $\gamma^-$, $\gamma^+$, that  $\rho^\prime$ as above is the canonically associated type 2 zero, and not some stray type 1 or  type 0.  

As we saw in \eqref{Eq:last}, when $\rho^\prime$ is near the critical line, $1/2+i \gamma^\prime$ is very near the midpoint of $1/2+i\gamma^\pm$ and so the argument of $\eta(1/2+i\gamma^\prime)$ will be very near to either $0$ or $\pi$.  Now consider the horizontal line segment joining $1/2+i\gamma^\prime$ to $\beta^\prime+i\gamma^\prime$.  If $\rho^\prime$ is not the type 2 zero on the contour, note that  $\rho^\prime$ can not lie inside (i.e.\ to the left of) the contour $\tre(\eta)=0$ connecting $1/2+i\gamma^-$ and $1/2+i\gamma^+$: the geometry of the level curves does not make sense.  Thus the horizontal line segment would have to cross this contour, and the argument of $\eta(s)$ would have to change by more than $\pi/2$ between $1/2+i\gamma^\prime$ and $\beta^\prime+i\gamma^\prime$.  We recover the change in argument by integrating
\[
\frac{d \arg(\eta(\sigma+i\gamma^\prime))}{d\sigma}
 = \tim\frac{\eta^\prime}{\eta}\left(\sigma+i \gamma^\prime\right)
\]
from $\sigma=1/2$ to $\sigma=\beta^\prime$.

We write
\[
\frac{\eta^\prime}{\eta}\left(\sigma+i \gamma^\prime\right)=\frac{h^\prime}{h}\left(\sigma+i \gamma^\prime\right)+\frac{\zeta^\pp}{\zeta^\prime}\left(\sigma+i \gamma^\prime\right),
\]
and via Stirling's formula $\tim\  h^\prime/h\left(\sigma+i \gamma^\prime\right)=\pi/4+O\left(1/\gamma^\prime\right)$.  Let $\rho^\pp$ be the zero of $\zeta^\pp(s)$ canonically associated to $\rho^\prime$.  With $s=\sigma+i\gamma^\prime$ we use \eqref{Eq:Marden} for $\zeta^\pp/\zeta^\prime(s)$ to see that
\begin{multline}\label{Eq:darg}
\frac{d \arg(\eta(\sigma+i\gamma^\prime))}{d\sigma}=\frac{\pi}{4}+\frac{\gamma^\prime-\gamma^\pp}{|\rho^\prime-\rho^\pp|^2}+\\
\sum_{\lambda^\prime\ne \rho^\prime}\tim\left(\frac{\rho^\pp-s}{(s-\lambda^\prime)(\rho^\pp-\lambda^\prime)}\right)
+O\left(\frac{1}{\gamma^\prime}\right)+O\left(\frac{1}{\gamma^\pp}\right).
\end{multline}
It will suffice to show the integrand \eqref{Eq:darg} multiplied by $\beta^\prime-1/2$ is $o(1)$.  We can eliminate the $\pi/4$ term and the $O$ terms, and based on the cancellation in the sum over $\lambda^\prime\ne\rho^\prime$, that term is also negligible. So we bound the integral by
\[
(\beta^\prime-1/2)\frac{\gamma^\prime-\gamma^\pp}{|\rho^\prime-\rho^\pp|^2}
\]
From the discussion following Theorem \ref{Thm:fund}, we expect  when $\beta^\prime-1/2$ is small, that $|\rho^\prime-\rho^\pp|\sim 2/\log(\gamma^\prime)$.    Since $\beta^\prime-1/2$ is $o(1/\log(\gamma^\prime))$, it suffices that $\gamma^\prime-\gamma^\pp$ is $O(1/\log(\gamma^\prime))$.  This is certainly true since this is the average vertical spacing of the zeros of both $\zeta^\prime(s)$ and $\zeta^\pp(s)$.
\end{proof}

\section{Appendix I: Improved Zhang-Ge Lemma}\label{S:IZL}
Let
\[
F(t)\overset{\text{def.}}=-\tre\frac{\eta^\prime}{\eta}\left(\frac12+i t\right).
\]
Let $\log\eta(s)$ be any choice of the branch of the logarithm in an open set which contains the critical line but no zeros of $\zeta^\prime$.  By the Cauchy-Riemann equations, 
\[
F(t)=-\frac{d \arg(\eta(1/2+it))}{dt}.
\]
\begin{lemma}
For $t>3$, $F(t)>0$.
\end{lemma}
\begin{proof}
In \cite[(2.9)]{Zhang}, Zhang deduces
\[
F(t)=-\sum_{\lambda^\prime}\tre\left(\frac{1}{1/2+it-\lambda^\prime}\right)+O\left(1\right)
\]
(where $\lambda^\prime$ denotes complex zeros of $\zeta^\prime(s)$.)\ \ 
In \cite[Lemma 7]{FanGe2}, Ge improves this to get
\begin{equation}\label{Eq:FanGe}
F(t)=-\sum_{\lambda^\prime}\tre\left(\frac{1}{1/2+it-\lambda^\prime}\right)+\log(2)/2+O\left(1/t\right).
\end{equation}
In fact the $O(1/t)$ error term is explicitly (writing $s=1/2+it$)
\[
-\tre\left(\frac{1}{s}-\frac{2}{s-1}\right)-\tre\sum_{n=1}^\infty\left(\frac{1}{s+a_n}-\frac{1}{s+2n}\right).
\]
Here
\[
-(2n+2)<-a_n<-2n
\]
is the unique real zero on $\zeta^\prime(s)$ in the interval.  For $s=1/2+it$,
\[
-\tre\left(\frac{1}{s}+\frac{2}{s-1}\right)=\frac{2}{1+4t^2}.
\]
We claim that the tail of the series,
\[
-\tre\sum_{2n>t}^\infty\left(\frac{1}{s+a_n}-\frac{1}{s+2n}\right)>0,
\]
as this sum is
\[
\sum_{2n>t}\tre\left(\frac{a_n-2n}{(s+a_n)(s+2n)}\right)=
\sum_{2n>t}(a_n-2n)\frac{\tre\left((s+a_n)(s+2n)\right)}{|s+a_n|^2|s+2n|^2}.
\]
With $2n<a_n<2n+2$, and $2n>t$, every term is positive.
Meanwhile
\[
-\tre\sum_{2n<t}^\infty\left(\frac{1}{s+a_n}-\frac{1}{s+2n}\right)=
\sum_{2n<t}\frac{-a_n-1/2}{|s+a_n|^2}+\frac{1/2+2n}{|s+2n|^2}.
\]
From $|s+a_n|^2>|s+2n|^2$ we deduce
\[
\frac{-1/2-a_n}{|s+a_n|^2}>\frac{-1/2-a_n}{|s+2n|^2},
\]
so this sum is bounded below by
\[
\sum_{2n<t}\frac{-a_n+2n}{|s+2n|^2}>-2\sum_{2n<t}\frac{1}{(2n+1/2)^2+t^2}.
\]
The sum has $t/2$ terms, each less than $1/t^2$ so this sum is bounded below by $-1/t$.  We conclude that
\[
F(t)>-\sum_{\beta^\prime>1/2}\tre\frac{1}{1/2+it-\lambda^\prime}+\log(2)/2+\frac{2}{1+4t^2}-\frac{1}{t},
\]
and for $t>3$,
\[
\log(2)/2+\frac{2}{1+4t^2}-\frac{1}{t}>0.
\]
\end{proof}
By Stirling's formula we have that for $\rho=1/2+i\gamma$ a zero of $\zeta(s)$, Zhang's Lemma 3 is, more explicitly,
\begin{equation}\label{Eq:ZLemma3}
F(\gamma)=\frac12\log\left(\frac{\gamma}{2\pi}\right)+O\left(\frac1t\right).
\end{equation}
Zhang's Lemma 4 becomes
\begin{lemma} For $n\ge 1$,
\[
\int_{\gamma_n}^{\gamma_{n+1}}F(t)\, dt=\pi.
\]
\end{lemma}

\section{Appendix II: Algorithms}\label{S:numerics}
In \emph{Mathematica} we computed the types of 108043 zeros of $\zeta^\prime$ in the range $10^6\le t\le 10^6+6\cdot 10^4$, previously located by Farmer \cite{Duenez}, and the types of the corresponding zeros $\rho_k$ of $\zeta$, $1747145\le k\le 1861805$.  

The algorithm followed the level curve $\tre(\eta)=0$ (labeled with the sign of $\tim(\eta)$) from each zero of $\zeta$ until it terminated in a zero of $\zeta^\prime$.  This was done with a small square (starting with size $1/5$ the distance to the nearest zero of $\zeta$), oriented in the direction of the unit tangent to the level curve, with the zero of $\zeta$ at the midpoint of a side.  A sign change of $\tre(\eta)$ at the two ends of a side is a sufficient, but not necessary condition for the level curve to cross a side of the square.  So instead we evaluated $\eta$ at the four corners and four midpoints, and used Lagrange interpolation with a quadratic polynomial to look for crossings on each of the sides.  The next square was again oriented in the direction of the unit tangent, with the approximation to the previous crossing at the midpoint of a side.  Whenever more than one labeled contour crossed a side, we halved the size of the square and recomputed.    When the labeled contour did not exit the square, we counted zeros of $\eta$ in the square via the Argument Principle.  If more than one zero of $\eta$ was in the square, again we halved the size of the square and recomputed.  Finally a numerical value for the zero of $\zeta^\prime$ in the square was determined with \emph{Mathematica}'s \texttt{FindRoot}.

Since the contour $\tre(\eta)=0$ is invariant when $\eta$ is rescaled by a positive real number, it was convenient to compute only the argument of the Gamma factor in the definition of $\eta$, via Stirling's formula, avoiding the numerical challenges of the exponential decay in the modulus.

Derivatives of $\zeta$ high in the critical strip are expensive to compute, and for all but the final step above, simple numerical approximations \cite[\S 5.7]{recipes} were sufficient.  Only  \texttt{FindRoot}  above used \emph{Mathematica}'s internal algorithms to compute $\zeta^\prime$.

After determining the zero of $\zeta^\prime$ at the termination of each of the labeled contours beginning at each of the zeros of $\zeta$, we knew the types of all the zeros involved.  As a  check, the algorithm was run on the zeros up to $T=1000$, and the results agreed with the classification done by visual inspection.

\subsubsection*{Acknowledgments}  We would like to thank both Rick Farr for sharing his computation of zeros of $\zeta^\prime(s)$ in the range $t<1000$, and David Farmer for sharing his computations in the range $10^6\le t\le 10^6+6\cdot 10^4$.  Thanks also to Fan Ge for suggesting a simpler argument with less explicit computation for the Improved Zhang Lemma, and helpful comments on the manuscript.  Thanks to UCSB Chancellor Henry Yang for providing summer research leave while the author was Dean of Undergraduate Education.

\end{document}